\documentclass[12pt, twoside, leqno]{article}
\usepackage{amsmath,amsthm}
\usepackage{amssymb,latexsym}
\usepackage{enumerate}
\usepackage{hyperref,amsbsy}

\newtheorem{thm}{Theorem}[section]
\newtheorem{cor}[thm]{Corollary}
\newtheorem{lem}[thm]{Lemma}
\newtheorem{hip}[thm]{Conjecture}

\theoremstyle{definition}
\newtheorem{defin}[thm]{Definition}
\newtheorem{rem}[thm]{Remark}

\begin{document}

\title{On the zero-sum constant, the~Davenport constant and their analogues.}

\author{Maciej Zakarczemny (Cracow)}

\date{}

\maketitle

\renewcommand{\thefootnote}{}

\footnote{2010 \emph{Mathematics Subject Classification}: Primary
11P70; Secondary 11B50.}

\footnote{\emph{Key words and phrases}: zero-sum sequence, Davenport constant, finite abelian group}

\renewcommand{\thefootnote}{\arabic{footnote}}
\setcounter{footnote}{0}

\begin{abstract}Let $D(G)$ be the Davenport constant of a finite Abelian group $G$. 
For a~positive integer $m$ (the case $m = 1$, is the classical one) let ${\mathsf  E}_m(G)$ (or $\eta_m(G)$, respectively) 
be the least positive integer $t$ such that every sequence of length $t$ in $G$ contains $m$ disjoint \mbox{zero-sum} sequences, 
each of length $|G|$ (or of length $\le exp(G)$ respectively). In this paper, we prove that if $G$ is an~Abelian group, 
then ${\mathsf E}_m(G)=D(G)-1+m|G|$, which generalizes Gao's relation. We investigate also the non-Abelian case. 
Moreover, we examine the asymptotic behavior of the sequences $({\mathsf E}_m(G))_{m\ge 1}$ and $(\eta_m(G))_{m\ge 1}.$ 
We prove a~generalization of Kemnitz's conjecture. The paper also contains a~result of independent interest, 
which is a stronger version of a result by Ch.~Delorme, O.~Ordaz, D.~Quiroz (see \cite[Theorem 3.2]{DOrQui}). 
At the and we apply the Davenport constant to smooth numbers and make a natural conjecture in the non-Abelian case.

\end{abstract}

\section{Introduction}
\indent We will define and investigate some generalizations of the Davenport constant (see \cite{AlD, EEGKR, FS10, GGZ, GHK, JOAI, JOAII, KRA}).
Davenport's constant is connected with algebraic number theory as follows.
For an algebraic number field $K$, let  $\mathcal{O}_K$ be its ring of integers and $G$ the
ideal class group of $\mathcal{O}_K$. Let $x\in \mathcal{O}_K$ be an irreducible element. If $\mathcal{O}_K$ is a Dedekind
domain, then $x\mathcal{O}_K=\prod\limits_{i=1}^{r} \mathcal{P}_i$, where $\mathcal{P}_i$ are prime ideals in $\mathcal{O}_K$ (not necessarily distinct). The Davenport constant $D(G)$ is the maximal number of prime ideals $\mathcal{P}_i$
(counted with multiplicites) in the prime ideal decomposition of the integral ideal $x\mathcal{O}_K$ (see \cite{FH, JOAI}).\\
The precise value of the Davenport constant is known, among others, for \mbox{$p$-groups} and for groups of rank at most two.
The determination of $D(G)$ for general finite Abelian groups is an open question (see \cite{BGAA}).


\section{General notation}
${}^{}$\indent Let $\mathbb{N}$ denote the set of positive integers (natural numbers). We set \\
$[a,b]=\{x:a\le x\le b,\,\,x\in\mathbb{Z}\},$ where $a,b\in\mathbb{Z}.$
Our notation and terminology is consistent with \cite{RG}. Let $G$ be a non-trivial additive finite Abelian group. $G$ can be uniquely decomposed as a direct sum of cyclic groups
$C_{n_1}\oplus C_{n_2}\oplus \ldots \oplus C_{n_r}$ with natural numbers $1<n_1|\ldots|n_r.$ The number $r$ of summands in the above decomposition of $G$ is expressed as $r = r(G)$ and called the rank of $G.$ The integer $n_r$ is called the exponent of $G$ and denoted by $\exp(G).$ In addition, we define $D^*(G)$ as $D^*(G) =1+ \sum\limits_{i=1}^r (n_i- 1)$.\\
We write any finite sequence $S$ of $l$ elements of $G$ in the form $$\prod\limits_{g\in G}g^{\nu_g(S)}=g_1\cdot\ldots\cdot g_l,$$
(this is a formal Abelian product), where $l$ is the length of $S$ denoted by $|S|, \nu_g(S)$ is the multiplicity of $g$ in $S.$
$S$ corresponds to the sequence (in the traditional sense) $(g_1,g_2,\ldots,g_l)$ where we forget the ordering of the terms.
By $\sigma (S)$ we denote the sum of $S:$
$$\sigma (S)=\sum\limits_{g\in G}\nu_g(S)g\in G.$$
The Davenport constant $D(G)$ is defined as the smallest $t\in\mathbb{N}$ such that each sequence over $G$ of length at least $t$ has a non-empty zero-sum subsequence. Equivalently, $D(G)$ is the maximal length of a zero-sum sequence of elements of $G$ and with no proper zero-sum subsequence.
One of the best bounds for $D(G)$ known so far is:
\begin{equation}\label{xeq1}
D^*(G)\le D(G)\le n_r\left(1+\log{\tfrac{|G|}{n_r}}\right).
\end{equation}
Alford, Granville and Pomerance \cite{AGP} in 1994 used the bound (\ref{xeq1}), to prove the existence of infinitely many Carmichael numbers. Dimitrov \cite{Dim} used the Alon Dubiner constant (see \cite{AlD}) to prove the inequality:
$$
\tfrac{D(G)}{D^*(G)}\le (K r\log{r})^r,
$$
for an absolute constant $K$.
It is known that for groups of rank at most two and for $p$-groups, where $p$ is a~prime,
the left hand side inequality (\ref{xeq1}) is in fact an equality (see Olson \cite{JOAI}).
This result suggests that $D^*(G)=D(G).$ However, there are infinitely many groups $G$ with rank $r>3$ such that $D(G)>D^*(G)$.
There are more recent results on groups where the Davenport constant does not match the usual lower bound (see \cite{GSch}).
The following Remark \ref{th1} lists some basic facts for the Davenport constant (see \cite{DOrQui, GSch}).
\begin{rem}\label{th1} Let $G$ be a~finite additive Abelian group.
\begin{enumerate}
\item Then $D(G)=D^*(G)$ in each of the following cases:
\begin{enumerate}
\item $G$ is a~$p-$group,
\item $G$ has rank $r\le 2,$
\item $G=C_p\oplus C_{p}\oplus C_{p^nm},$ with $p$ a~prime number, $n\ge 2$ and $m$ an integer coprime
with $p^n$ (more generally, if $G=G_1\oplus C_{p^kn},$ where $G_1$ is a~$p-group$ and $p^k\ge D^*(G_1)),$
\item $G=C_2\oplus C_2\oplus C_2\oplus C_{2n},$ with odd $n,$
\item $G=C_{2p^{k_1}}\oplus C_{2p^{k_2}}\oplus C_{2p^{k_3}},$ with $p$ prime,
\item $G=C_2\oplus C_{2n}\oplus C_{2nm},$
\item $G=C_3\oplus C_{3n}\oplus C_{3nm}.$
\end{enumerate}
\item Then $D(G)>D^*(G)$ in each of the following cases:
\begin{enumerate}
\item $G=C_n\oplus  C_n\oplus C_n \oplus C_{2n} $ for every odd $n\ge 3,$
\item $G=C_3\oplus  C_9\oplus C_9 \oplus C_{18} ,$
\item $G=C_3\oplus  C_{15}\oplus C_{15} \oplus C_{30}.$
\item Let $n\ge 2,\,k\ge 2,\,(n,k)=1,0\le \rho \le n-1, $ and\\
$G=C_n^{(k-1)n+\rho}\oplus C_{kn}.$ If $\rho \ge 1$ and $\rho \not\equiv n\pmod k$, then $D(G)\ge D^*(G)+\rho.$ 
If $\rho\le n-2$ and $x(n-1-\rho)\not\equiv n\pmod k$ \\for any $x\in [1,n-1],$ then $D(G)\ge D^*(G)+\rho+1.$
\end{enumerate}
\end{enumerate}
\end{rem}

\section{Definitions}
In this section, we will provide some definitions of classical invariants.
We begin with some notation and remarks that will be used throughout the paper.
\begin{defin}\label{DD}
Let $G$ be a~finite Abelian group, $m$ and $k$~positive integers such that $k\ge \exp(G),$ and $\emptyset \neq I\subset \mathbb{N}.$\\
1) By $s_{I}(G)$ we denote the smallest $t\in\mathbb{N}\cup\{\infty\}$ such that every sequence $S$ over $G$ of length $t$ contains a~non-empty subsequence $S'$ such that $\sigma(S')=0$, $|S'|\in I$ (see \cite{RBGB, ChMG, DOrQui}).\\
We use notation $s_{\le k}(G)$ or $D^k(G)$ to denote $s_{I}(G)$ if $I=[1,k]$ (see \cite{RBGB,DOrQui}).\\
2) By $s_{I,m} (G)$ we denote the smallest $t\in\mathbb{N}\cup\{\infty\}$ such that every sequence $S$ over $G$ of length $t$ contains at least $m$ disjoint non-empty subsequences $S_1,S_2,\ldots,S_m$ such that $\sigma(S_i)=0$, $|S_i|\in I $.\\
3) Let ${\mathsf E}(G):=s_{\{|G|\}} (G),$ i.e. the smallest $t\in\mathbb{N}\cup\{\infty\}$ such that every sequence $S$ over $G$ with length $t$ contains non-empty subsequence $S'$ such that $\sigma(S')=0,\,|S'|=|G|.$\\
Note that ${\mathsf E}(G)$ is the classical zero-sum constant.\\
4) Let ${\mathsf E}_m(G):=s_{\{|G|\},m} (G),$ i.e. the smallest $t\in\mathbb{N}\cup\{\infty\}$ such that every sequence $S$ over $G$ with length $t$ contains at least $m$ disjoint non-empty subsequences
$S_1,S_2,\ldots,S_m$ such that $\sigma(S_i)=0,\,|S_i|=|G|, $ for $i\in [1,m].$\\
5) Also, we define $\eta (G):=s_{\le \exp(G)},\,s(G):=s_{\{\exp(G)\}}(G),\,D_m(G):=s_{\mathbb{N},m} (G)$ (see \cite{FH}), $s_{(m)}(G):=s_{\{\exp(G)\},m}(G).$
\end{defin}
\begin{rem}\label{th9c2}
For $n,n'\in\mathbb{N},\,\emptyset\neq I\subseteq \mathbb{N},$ by definition $$s_{\mathbb{N}}(G)=s_{[1,D(G)]}(G)=D(G)=D_1(G),$$
$$s_{I,1}(G) = s_{I}(G),\,\,s_{[1,D(G)],n}(G)=D_n(G),\,s_{\{|G|\},n}=E_n(G).$$
Note that $s(G)=s_{\{exp(G)\}}(G)$ is
the classical Erd\"os - Ginzburg - Ziv constant (see \cite{YGZ}).
We call $s_{(m)}(G)=s_{\{\exp(G)\},m}(G)$ the $m$-wise Erd\"os--Ginzburg--Ziv constant of $G.$ Thus, in this notation $s_{(1)}(G)=s(G).$
\end{rem}
One can derive, for example, the following inequalities.
\begin{rem}\label{th9c2b}
If $n'\ge n$, then $s_{I,n'}(G)\ge s_{I,n}(G).$\\
If $D(G)\ge k\ge k'\ge \exp(G)$, then $$s_{[1,\exp(G)],n}(G)\ge s_{[1,k' ],n}(G) \ge s_{[1,k],n}(G)\ge D_n(G).$$
If $D(G)\ge k\ge\exp(G)$, then $\eta(G) \ge s_{\le k}(G)\ge D(G).$\\
If $k\ge k'$, then
\begin{equation}\label{xeq31nnn}
s_{\le k'}(G) \ge s_{\le k}(G).
\end{equation}
We note that a~sequence $S$ over $G$ of length $|S|\ge mD(G)$ can be partitioned into $m$ disjoint subsequences $S_i$ of length $|S_i|\ge D(G).$ Thus, each $S_i$ contains a~non-empty zero-sum subsequence. Hence, $D_m(G)\le mD(G).$  See also \cite{FH} Proposition 1 (ii).
\end{rem}

\section{The $m$--wise zero-sum constant and the \mbox{$m$--wise} Erd\"os--Ginzburg--Ziv constant of $G$.}
In 1996, Gao and Caro proved independently that:
\begin{equation}\label{ED}
E(G)=D(G)+|G|-1,
\end{equation}
for any finite Abelian group (see \cite{WGAO2, CAR}]; for a proof in modern language we refer to \cite[Proposition 5.7.9]{GHK}; see also \cite{WGAO, DOrQui, YHO, GWS}).
Relation \eqref{ED} unifies research on constants $D(G)$ and $E(G).$
We start this section with the result that can be used to unify research on constants $D(G)$ and $E_m(G).$
\begin{thm}\label{zsml1}
If $G$ is a~finite Abelian group of order $|G|$, then
\begin{equation}
E_m(G)=E(G)+(m-1)|G|=D(G)+m|G|-1.
\end{equation}
\end{thm}
\begin{proof}
By \eqref{ED} we obtain $E(G)+(m-1)|G|=D(G)+m|G|-1$.\\
Let $S=g_1g_2\cdot\ldots\cdot g_{{}_{D(G)-1}}$ be a~sequence of $D(G)-1$ non-zero elements in $G.$
Using the definition of $D(G)$, we may assume that $S$ does not contain any non-empty subsequence $S'$ such that $\sigma(S')=0$.
We put
$$T=a_1\cdot a_2\cdot\ldots\cdot a_{{}_{D(G)-1}} \cdot\underbrace{0\cdot\ldots\cdot 0}_{m|G|-1\, \mathrm{times}}.$$
We observe that the sequence $T$ dose not contain $m$ disjoint non-empty subsequences $T_1,T_2,\ldots,T_m$ such that $\sigma(T_i)=0$ and $|T_i|=|G|$ for $i\in[1,m].$ This implies that $E_m(G)>D(G)+m|G|-2.$ Hence, $E_m(G)\ge E(G)+(m-1)|G|.$ On the other hand, if $S$ is any sequence such that \mbox{$|S|\ge E(G)+ (m-1)|G|,$} then one can sequentially extract at least $m$ disjoint subsequences $S_1,\ldots,S_m,$ such that $\sigma(S_i)=0$ in $G$ and $|S_i|=|G|.$ Thus, \mbox{$E_m(G)\le E(G)+(m-1)|G|$.}
\end{proof}
We anticipate that Theorem \ref{zsml1} can be used to obtain a generalization of the classical Theorem of Hall (see \cite[Section 3]{Hall}).
\begin{cor}
For every finite Abelian group the sequence $({\mathsf E}_m(G))_{m\ge 1}$ is an arithmetic progression with difference $|G|.$
\end{cor}
\begin{cor}
If $p$ is a~prime and $G=C_{{}_{p^{e_1}}}\oplus\ldots\oplus C_{{}_{p^{e_k}}}$ is a~$p-$group, then for a~natural $m$ we have:
$$
E_m(G)= mp^{{\sum_{i=1}^k e_i}}+\sum\limits_{i=1}^k(p^{e_i}-1).
$$
\end{cor}
\begin{proof}
It follows from Remark \ref{th1} and Theorem \ref{zsml1}.
\end{proof}
We recall that by $s_{(m)}(G)$ we denote the smallest $t\in\mathbb{N}\cup\{\infty\}$ such that every sequence $S$ over $G$ of length $t$ contains at least $m$ disjoint non-empty subsequences $S_1,S_2,\ldots,S_m$ such that $\sigma(S_i)=0,\,|S_i|=\exp(G)$.
\begin{thm}\label{zsml12}
If $G$ is a~finite Abelian group, then
$$
\eta(G)+m\exp(G)-1\le s_{(m)}(G)\le s(G)+(m-1)\exp(G).
$$
\end{thm}
\begin{proof}
The proof runs along the same lines as the proof of Theorem \ref{zsml1}.\\
Let $S=g_1g_2\cdot\ldots\cdot g_{{}_{\eta(G)-1}}$ be a~sequence of $\eta(G)-1$ non-zero elements in $G.$
Using the definition of $\eta(G)$, we may assume that $S$ does not contain any non-empty subsequence $S'$ such that $\sigma(S_i)=0,\,|S_i|\le \exp(G)$.
We put
$$
T=a_1\cdot a_2\cdot\ldots\cdot a_{{}_{\eta(G)-1}} \cdot\underbrace{0\cdot\ldots\cdot 0}_{m\exp(G)-1\, \mathrm{times}},
$$
We observe that the sequence $T$ not contain $m$ disjoint non-empty subsequences $T_1,T_2,\ldots,T_m$ such that $\sigma(T_i)=0$ and $|T_i|=\exp(G)$ for\\ $i\in[1,m].$ This implies that $s_{(m)}(G)>\eta(G)+m\exp(G)-2.$ Hence,\\
$
s_{(m)}(G)\ge \eta(G)+m\exp(G)-1.
$ On the other hand, if $S$ is any sequence over $G$ such that $|S|\ge s(G)+(m-1)|G|,$ then one can sequentially extract at least $m$ disjoint subsequences $S_1,\ldots,S_m,$ such that $\sigma(S_i)=0$ in $G$ and $|S_i|=\exp(G).$ Thus,
$
s_{(m)}(G)\le s(G)+(m-1)\exp(G).$
\end{proof}
It was conjectured by Gao that for every finite Abelian group $G,$ one has $\eta(G)+\exp(G)-1=s(G)$ (see \cite[Conjecture
6.5]{GGZ}). If this conjecture is true, then by Theorem \ref{zsml12} for every finite Abelian group $G$ the equality $s_{(m)}(G)=s(G)+(m-1)\exp(G)$ holds, i.e. the sequence $(s_{(m)}(G))_{m\ge 1}$ is an arithmetic progression with difference $\exp(G).$ We will note that the equation $\eta(G)+\exp(G)-1=s(G)$ is true for all finite Abelian groups of rank at most two (see \cite[Theorem 2.3]{BGSCH}).
\begin{cor}
If $G=C_{n_1}\oplus C_{n_2},$ where $n_1|n_2,$ then for a~natural $m$ we have:
\begin{equation}\label{ttrr1}
E_m(G)= mn_2n_1+n_2+n_1-2,
\end{equation}
\begin{equation}\label{ttrr2}
D_m(G)=mn_2+n_1-1,
\end{equation}
\begin{equation}\label{ttrr3}
s_{(m)}(G)= (m+1)n_2+2n_1-3,
\end{equation}
\begin{equation}\label{ttrr4}
s_{(m)}(G)-D_m(G)=D(G)-1.
\end{equation}
\end{cor}
\begin{proof}
The equation \eqref{ttrr1} is a~consequence of Remark \ref{th1} and Theorem \ref{zsml1}, the equation \eqref{ttrr2} follows from Proposition 5 in \cite{FH}. By applying \cite[Theorem 2.3]{BGSCH} and Theorem \ref{zsml12}, we can obtain \eqref{ttrr3}.
The equation \eqref{ttrr4} is a~consequence of equations \eqref{ttrr2} and \eqref{ttrr3}. Note that, if $G$ is an Abelian group, then $D(G)-1$ is the maximal length of a~zero-sum free sequence over $G.$
\end{proof}

\section{A generalization of the Kemnitz conjecture}
Kemnitz's conjecture states that every set $S$ of $4n-3$ lattice points in the plane has a subset $S'$ with $n$ points whose centroid is also a lattice point. The conjecture was proved by Christian Reiher (see \cite{REI}).
In order to prove the generalization of this theorem, we will use the equation \eqref{ttrr3}. 
\begin{thm}\label{GK}
Let $n$ and $m$ be natural numbers. Let $S$ be a set of \\
$(m+3)n-3$ lattice points in $2$-dimensional Euclidean space.
Then there are at least $m$ pairwise disjoint sets $S_1,S_2,\ldots,S_m\subseteq S$ with $n$ points each, such that the centroid of each set $S_{i}$ is also a lattice point.
 \end{thm}
\begin{proof}
As Harborth has already noted (see \cite{EEGKR,HAR}), $s(C_n^r)$ is the smallest integer $l$ such that every set
of $l$ lattice points in $r$-dimensional Euclidean space contains $n$ elements which have a centroid in a lattice point.
By analogy, $s_{(m)}(C_n^r)$ is the smallest integer $l$ such that every set $S$ of $l$ lattice points in $r$-dimensional Euclidean space have $m$ pairwise disjoint subsets $S_1,S_2,\ldots,S_m$ each of cardinality $n,$ which centroids are also a lattice points.\\
Finally, in the case of $2$-dimensional Euclidean space, it is sufficient to use the equation \eqref{ttrr3}.
\end{proof}
\begin{rem} 
As we can see in point 3) of Definition \ref{DD}, 
Theorem \ref{GK} is also true when we replace \it{sets} with \it{multisets}.
\end{rem}
\section{Some results on $s_{I,m}(G)$ constant}
In this section, we will investigate zero-sum constants for finite Abelian groups. We start with $s_{\le k}(G),D_m(G) ,\eta(G)$ constants. Our main result of this section is Theorem \ref{mainth}.
Olson calculated $s_{\le p}(C_p^2)$ for a~prime number $p$ (see \cite{JOAII}).
No precise result is known for $s_{\le p}(C_p^n),$ where $n\ge 3.$
We need two technical lemmas:
\begin{lem}\label{th8b2}
Let $p$ be a~prime number and $n\ge 2.$ Then:
\begin{equation}
s_{\le (n-1)p}(C_p^n)\le (n+1)p-n.
\end{equation}
\end{lem}
\begin{proof}
Let $g_i\in C_p^n,\,i\in[1,(n+1)p-n].$ Embed $C_p^n$ into an Abelian group $F$ group isomorphic to $ C_p^{n+1}.$ Let $x\in F,\,x\notin C_p^n.$ Since $D(C_p^{n+1})=(n+1)p-n$ (see \cite{JOAI} or \mbox{Remark \ref{th1},(a)}) there exists a zero-sum subsequence $\prod\limits_{i\in I}(x+g_i)$ of the sequence $\prod\limits_{i=1}^{(n+1)p-n}(x+g_i).$ But this possible only if $p$ divides $|I|.$ Rearranging subscripts, we may assume that $g_1+g_2+\ldots+g_{ep}=0,$ where $e\in[1,n].$ We are done if $e\in[1,n-1].$ If $e=n$ we obtain a zero-sum sequence $S=g_1\cdot g_2\cdot\ldots\cdot g_{np}.$ Zero-sum sequence~$S$ contains a~proper zero-sum subsequence $S',$ since $D(C_p^n)=np-(n-1),$ and thus zero-sum subsequence of length not exceeding $\lfloor\frac{np+1}{2}\rfloor \le (n-1)p.$
\end{proof}

\begin{cor}\label{th8b2c}
Let $p$ be a~prime. Then:
\begin{equation}
s_{\le 2p}(C_p^3)\le 4p-3.
\end{equation}
\end{cor}
\begin{lem}\label{th9d}
Let $G$ be a~finite Abelian group, $k\in\mathbb{N},\,k\ge \exp(G)$.\\
If $s_{[1,k],1}(G)\le s_{[1,k],m}(G)+k$, then $s_{[1,k],m+1}(G)\le s_{[1,k],m}(G)+k.$
\end{lem}
\begin{proof}
Let $S$ be a~sequence over $G$ of length $s_{[1,k],m}(G)+k$. The sequence $S$ contains a non-empty subsequence $S_0|S$ such that $\sigma(S_0)=0.$ Then $|S_0|\in[1,k],$ since $|S|\ge s_{[1,k],1}(G)=s_{\le k}(G).$ By definition of $s_{[1,k],m}(G)$ the remaining elements in $S$ contain $m$ disjoint non-empty subsequences $S_i|S$ such that $\sigma(S_i)=0$, $|S_i|\in[1,k],$ where $i\in[1,m].$ Thus, we get $m+1$ non-empty disjoint subsequences $S_i|S$ such that $\sigma(S_i)=0$, $|S_i|\in[1,k],$ where $i\in[0,m].$
\end{proof}

\begin{cor}\label{th9e}
Let $G$ be a~finite Abelian group, $k\ge \exp(G).$\\
If $s_{[1,k],1}(G)\le s_{[1,k],m}(G)+k,$ then $s_{[1,k],m+n} (G)\le s_{[1,k],m} (G)+nk.$
\end{cor}

\begin{proof}
We use Lemma~\ref{th9d} and Remark \ref{th9c2}.
\end{proof}

\begin{cor}\label{th9f}
Let $G$ be a~finite Abelian group, $k\ge \exp(G)$. Then:
\begin{equation}\label{xeq32}
D_{n} (G)\le s_{[1,k],n}(G)\le s_{\le k}(G)+k(n-1),
\end{equation}
\begin{equation}\label{xeq34}
D_{n} (G)\le s_{[1,\exp(G)],n}(G)\le \eta(G)+\exp(G)(n-1).
\end{equation}
\end{cor}
\begin{proof}
We use Remark \ref{th9c2b}, Corrolary \ref{th9e} with $m=1$ and get (\ref{xeq32}).\\
We put $k=\exp(G)$ in (\ref{xeq32}) and get (\ref{xeq34}).
\end{proof}

\begin{rem}\label{th9}
It is known that:\\
$\eta(C_n^3)=8n-7,$ if $n=3^\alpha 5^\beta,$ with $\alpha,\beta\ge 0;$\\$ \eta(C_n^3)=7n-6$, if $n=2^\alpha 3$, with $\alpha\ge 1;$\\$\eta(C_2^3)=8;\eta(C_3^3)=17;\eta(C_3^4)=39;\eta(C_3^5)=89;\eta(C_3^6)=223,$ (see \cite{BGAA}).
\end{rem}

\begin{cor}\label{th10}
We have that:\\
$D_{m}(C_n^3)\le nm+7n-7,$ if $n=3^\alpha 5^\beta,$ with $\alpha,\beta\ge 0;$
\\$ D_{m}(C_n^3)\le nm+6n-6$, if $n=2^\alpha 3$, with $\alpha\ge 1;$\\
$D_{m}(C_3^4)\le 3m+36; D_{m}(C_3^5)\le 3m+86;D_{m}(C_3^6)\le 3m+220.$
\end{cor}
\begin{proof}
By Corollary \ref{th9f} and Remark \ref{th9}.
\end{proof}

In the next Lemma we collect several useful properties on the Davenport constant.
\begin{lem}\label{th7}
Let $G$ be a~non-trivial finite Abelian group and $H$ subgroup of $G.$ Then:
\begin{equation}\label{xeq3}
D(H)+D(G/H)-1\le D(G)\le D_{D(H)}(G/H) \le D(H)D(G/H).
\end{equation}
\end{lem}

\begin{proof}
The inequality $D(H)+D(G/H)-1\le D(G)$ is proved in Proposition 3(i) in \cite{FH}. Now we prove the inequality $D(G)\le D_{D(H)}(G/H)$ on the same lines as in \cite{DOrQui} (we include the proof for the sake of completeness). If $|S|\ge D_{D(H)}(G/H)$ is any sequence over $G$, then one can, by definition, extract at least $D(H)$ disjoint non empty subsequences $S_1,\ldots,S_{D(H)}|S$ such that $\sigma(S_i)\in H.$ Since $T=\prod\limits_{i=1}^{D(H)}\sigma (S_i)$ is a~sequence over $H$ of length $D(H),$
thus there exists a~non-empty subset $I\subseteq [1,D(H)]$ such that $T'=\prod\limits_{i\in I}\sigma (S_i)$ is a~zero-sum subsequent of $T.$\\
We obtain that $S'=\prod\limits_{i\in I}S_i$ is a~non-empty zero-sum subsequence of $S$.\\
The inequality $D_{D(H)}(G/H) \le D(H)D(G/H)$ follows from Remark \ref{th9c2b}.
\end{proof}

\begin{thm}\label{th16}
For an Abelian group $C_{p}\oplus C_{n_2}\oplus C_{n_3}$ such that $p|n_2|n_3\in\mathbb{N},$ where $p$ is a prime number, we have:
\begin{equation}\label{xeq10}
n_3+n_2+p-2\le D(C_{p}\oplus C_{n_2}\oplus C_{n_3})\le D_{\tfrac{n_2}{p}+\tfrac{n_3}{p}-1}(C_{p}^3)\le 2n_3+2n_2-3.
\end{equation}
\end{thm}

\begin{proof}
If $G=C_{p}\oplus C_{n_2}\oplus C_{n_3}$ such that $p|n_2|n_3\in\mathbb{N},$ then $\exp(G)=n_3.$
Note that: $n_3+n_2+p-2=D^*(G)\le D(G).$ Denoting by $H$ a subgroup of $G$ such that
$ H\cong C_{\tfrac{n_2}{p}}\oplus C_{\tfrac{n_3}{p}}.$
The quotient group $G/H\cong C_{p}\oplus C_{p}\oplus C_{p}.$\\
By Lemma \ref{th7} we get
\begin{equation}\label{xeq12}
D(G)\le D_{D(H)}(G/H)=D_{\tfrac{n_2}{p}+\tfrac{n_3}{p}-1}(C_{p}^3).
\end{equation}
By (\ref{xeq32}) (with $m={\tfrac{n_2}{p}+\tfrac{n_3}{p}-1},\,k=2p$) and Corollary \ref{th8b2c}, we get
\begin{equation}\label{xeq23}
\begin{split}
D(G)&\le s_{\le 2p}(C_{p}^3)+ 2p(\tfrac{n_2}{p}+\tfrac{n_3}{p}-2)\le \\
&\le 2p(\tfrac{n_2}{p}+\tfrac{n_3}{p}-2)+(4p-3)= 2n_3+2n_2-3.
\end{split}
\end{equation}
\end{proof}
Our next goal is to generalize \cite[Theorem 3.2]{DOrQui} to the case $r(G)\ge 3.$
\begin{thm}\label{thth}
Let $H, K$ and $L$ be Abelian groups of orders $|H| = h,$\\
$|K| = k$ and $|L|=l$. If $G = H\oplus K\oplus L$ with $h|k|l$. \\
Let $\Omega(h)$ denote the total number of prime factors of $h.$ Then:
\begin{equation}
s_{\le 2^{\Omega(h)}l}(G)\le 2^{\Omega(h)}(2l+k + h) - 3.
\end{equation}
\end{thm}

\begin{proof}
The proof will be inductive.
If $h=1,$ then by \cite[Theorem 3.2]{DOrQui} we have:
\begin{equation}
s_{\le 2^{\Omega(h)}l}(G)=s_{\le l}(K\oplus L)\le 2l+k-2= 2^{\Omega(1)}(2l+k+1)-3.
\end{equation}
Assume that $h>1$ and let $p$ be a prime divisor of $h.$\\
Let $H_1$ be a subgroup of $H,$ $K_1$ be a subgroup of $K,$ $L_1$ be a subgroup of $L,$ with indices $[H:H_1]=[K:K_1]=[L:L_1]=p.$\\
Put $h=ph_1,\,k=pk_1,\,l=pl_1$ and $Q=H_1\oplus K_1\oplus L_1.$\\
Assume inductively that theorem is true for $Q$ i.e.
\begin{equation}
s_{\le 2^{\Omega(h_1)}l_1}(Q)\le 2^{\Omega(h_1)}(2l_1+k_1+h_1)-3.
\end{equation}
Let $s=2^{\Omega(h)}(2l+k+h)-3$ and $S=g_1g_2\cdot\ldots\cdot g_s$ be a sequence of $G.$\\
We shall prove that there exists a subsequence of $S$ with length smaller or equal to $2^{\Omega(h)}l$ and zero sum. Let $b_i=g_i+Q\in G/Q.$ We consider the~sequence $\prod\limits_{i=1}^s b_i$ of length~$s.$\\
The quotient group $G/Q$ is isomorphic to $C_p^3$ and
\begin{equation}\label{cos2}
s=2p\big(2^{\Omega(h_1)}(2l_1+k_1+h_1)-2\big)+4p-3.
\end{equation}
Therefore, by Lemma \ref{th8b2c} there exist pairwise disjoint sets $I_j \subset [1,s],$\\
$|I_j|\le~2p$ and
\begin{equation}\label{cos1}
j\le j_0=2^{\Omega(h_1)}(2l_1+k_1+h_1)-1
\end{equation}
such that each sequence $\prod\limits_{i\in I_j}b_i$ has a zero sum in $G/Q.$\\
In other words $\sigma(\prod\limits_{i\in I_j}g_i)=\sum\limits_{i\in I_j}g_i\in Q.$
By induction assumption for $Q$ there exists $J\subseteq[1,j_0]$ with $|J|\le 2^{\Omega(h_1)}l_1$ such that
$\sum\limits_{j\in J}\sigma(\prod\limits_{i\in I_j}g_i)=0.$
Thus, we obtain zero sum subsequence of $S$ in $G$ of length not exceeding
$\sum\limits_{j\in J}|I_j|\le 2^{\Omega(h_1)}l_1\cdot 2p=2^{\Omega(h)}l,$ which ends the inductive proof.
\end{proof}
\begin{thm}\label{mainth}
Let $H_1, H_2,\ldots,H_n$ be Abelian groups of orders $|H_i| = h_i$.\\
If $n\ge 2$ and $G = H_1\oplus H_2\oplus\ldots\oplus H_n$ with $h_1|h_2|\ldots|h_n$, then
\begin{equation}\label{ththe1}
s_{\le (n-1)^{\Omega( h_n)}h_n}(G)\le (n-1)^{\Omega (h_{n})}(2(h_n-1)+(h_{n-1}-1)+\ldots + (h_1-1)+1).
\end{equation}
\end{thm}
\begin{proof}
We proceed by induction on $n$ and $h_n.$\\
If $n=2$, then the inequality (\ref{ththe1}) holds by [\cite{DOrQui} Theorem 3.2].
Namely:
\begin{equation}
\begin{split}
s_{\le (2-1)^{\Omega(h_2)}h_2}(H_1\oplus H_2)=\\
=s_{\le h_2}(H_1\oplus H_2)\le 2h_2+h_1-2=(2-1)^{\Omega(h_2)}(2(h_2-1)+(h_1-1)+1).
\end{split}
\end{equation}
Suppose that the inequality (\ref{ththe1}) holds for fixed $n-1\ge 2$:
\begin{equation}\label{ththe2}
\begin{split}
s_{\le (n-2)^{\Omega( h_{n-1})}h_{n-1}}(H_1\oplus\ldots\oplus H_{n-1})\le
\\
\le (n-2)^{\Omega ( h_{n-1})}(2(h_{n-1}-1)+(h_{n-2}-1)+\ldots + (h_1-1)+1).
\end{split}
\end{equation}
If $n\ge 3$ and $h_1=1,$ then $G$ and $H_2\oplus\ldots\oplus H_n$ are isomorphic. By (\ref{xeq31nnn}):
\begin{equation}\label{ththe3b}
\begin{split}
s_{\le (n-1)^{\Omega( h_n)}h_{n}}(G)=\\
=s_{\le (n-1)^{\Omega( h_n)}h_{n}}(H_2\oplus\ldots\oplus H_n)\le\\
\le s_{\le (n-2)^{\Omega( h_n)}h_{n}}(H_2\oplus\ldots\oplus H_n).
\end{split}
\end{equation}
Thus by the induction hypothesis (\ref{ththe2}):
\begin{equation}
\begin{split}
s_{\le (n-1)^{\Omega( h_n)}h_{n}}(G)\le\\
(n-2)^{\Omega(h_{n})}(2(h_{n}-1)+(h_{n-1}-1)+\ldots + (h_2-1)+1)\le
\end{split}
\end{equation}
$$
\le (n-1)^{\Omega(h_{n})}(2(h_{n}-1)+(h_{n-1}-1)+\ldots + (h_2-1)+(1-1)+1).
$$
Therefore (\ref{ththe1}) holds.\\
Suppose that the inequality (\ref{ththe1}) holds for fixed $n\ge 3$ and fixed $h,$ such that $h_1>h\ge 1$:
\begin{equation}\label{ththe3m}
s_{\le (n-1)^{\Omega( h_n)}h_n}(G)\le (n-1)^{\Omega(h_{n})}(2(h_n-1)+(h_{n-1}-1)+\ldots + (h-1)+1).
\end{equation}
Let $p$ be a prime divisor of $h_1.$
Let $H_i^*$ be a subgroup of index $p$ of a group $H_i$.
Put $h_i=ph_i^*$ and $Q=H_1^*\oplus H_2^*\oplus\ldots\oplus H_n^*.$ By inductive assumption, the inequality (\ref{ththe1}) holds for $Q$:
\begin{equation}\label{ththe3n}
s_{\le (n-1)^{\Omega( h_n^*)}h_n^*}(Q)\le 2^{\Omega(h_{n}^*)}(2(h_n^*-1)+(h_{n-1}^*-1)+\ldots + (h_1^*-1)+1).
\end{equation}
We put $s=(n-1)^{\Omega(h_{n})}(2(h_n-1)+(h_{n-1}-1)+\ldots + (h_1-1)+1)$ and let $S=g_1g_2\cdot\ldots\cdot g_s$ be a sequence of $G.$\\
We shall prove that there exists a subsequence of $S$ with length smaller or equal to $(n-1)^{\Omega( h_{n})}h_n$ and zero sum. Let $b_i=g_i+Q,\,1\le i\le s,\,$ be the sequence of $G/Q.$
The quotient group $G/Q$ is isomorphic to $C_p^n$ and
\begin{equation}
\begin{split}
s=(n-1)^{\Omega(h_{n}^*)+1}\Big(p\big(2(h_n^*-1)+\ldots + (h_1^*-1)+1\big)+n(p-1)\Big)\ge \\
p(n-1)^{\Omega(h_{n}^*)+1}\big(2(h_n^*-1)+\ldots + (h_1^*-1)+1\big)+(n-1)n(p-1)\ge \\
(n-1)p\Big((n-1)^{\Omega(h_{n}^*)}\big(2(h_n^*-1)+\ldots + (h_1^*-1)+1\big)\Big)+2p-n.
\end{split}
\end{equation}
Therefore, by Lemma \ref{th8b2} there exist pairwise disjoint $I_j \subseteq [1,s]$ with $|I_j|\le (n-1)p$ and
\begin{equation}\label{wygn}
j\le j_0=(n-1)^{\Omega(h_{n}^*)}(2(h_n^*-1)+\ldots + (h_1^*-1)+1)
\end{equation}
such that sequence $\prod\limits_{i\in I_j}b_i$ has a zero sum in $G/Q.$\\
In another words $\sigma(\prod\limits_{i\in I_j}g_i)=\sum\limits_{i\in I_j}g_i\in Q.$
By induction (\ref{ththe1}) for $Q$ there exists $J\subseteq[1,j_0]$ with $|J|\le (n-1)^{\Omega( h_n^*)}h_n^*$ such that
$\sum\limits_{j\in J}\sigma(\prod\limits_{i\in I_j}g_i)=0.$\\
Thus, we obtain zero sum subsequence of $S$ in $G$ of length not exceeding $\sum\limits_{j\in J}|I_j|\le(n-1)^{\Omega(h_n^*)}h_n^*\cdot (n-1)p=(n-1)^{\Omega(h_n)}h_n.$

\end{proof}

\section{Davenport's constant and smooth numbers}
First, we recall the notion of a smooth number. Let $F=\{q_1,q_2,\ldots,q_r\}$ be a subset of positive integers.
A~positive integer $k$ is said to be smooth over a~set $F$ if $k=q_1^{e_1}\cdot q_2^{e_2}\cdot\ldots\cdot q_r^{e_r}$ where $e_i$ are non-negative integers.
\begin{rem}Let $n\in\mathbb{N}$.
Each smooth number over a~set $\{{q_1}^{n}, {q_2}^{n},\ldots, {q_r}^{n}\}$ is an~$n$-th power of a~suitable smooth number over the set
$\{q_1, q_2,\ldots, q_r\}.$
\end{rem}
\begin{defin}
Let $\{p_1,p_2,\ldots,p_r\}$ be a set of distinct prime numbers. By $c(n_1,n_2,\ldots, n_r),$ we denote the smallest $t\in \mathbb{N}$ such that every sequence $M$ of smooth numbers over a set $\{p_1,p_2,\ldots,p_r\}$, of length $t$ has a~non empty subsequence $N$ such that the product of all the terms of $N$ is a~smooth number over a~set $\{{p_1}^{n_1}, {p_2}^{n_2},\ldots, {p_r}^{n_r}\}.$
\end{defin}
\begin{thm}
Let $n_1,n_2,\ldots,n_r$ be integers such that $1<n_1|n_2|\ldots|n_r.$ Then:
\begin{equation}\label{eCD0}
c(n_1,n_2,\ldots, n_r)= D(\mathbb{Z}_{n_1}\oplus \mathbb{Z}_{n_2}\oplus \ldots \oplus \mathbb{Z}_{n_r}).
\end{equation}
\end{thm}
\begin{proof} It follows on the same lines as proof of \cite[Theorem 1.6.]{ChMG}. First we will prove that
\begin{equation}\label{eCD}
c(n_1,n_2,\ldots, n_r)\le D(\mathbb{Z}_{n_1}\oplus \mathbb{Z}_{n_2}\oplus \ldots \oplus \mathbb{Z}_{n_r}).
\end{equation}
 We put $l=D(\mathbb{Z}_{n_1}\oplus \mathbb{Z}_{n_2}\oplus \ldots \oplus \mathbb{Z}_{n_r}).$ Let $M=(m_1,m_2,\ldots,m_l)$ be a~sequence of smooth numbers with respect to $F=\{p_1,p_2,\ldots,p_r\}$.\\
For all $i\in[1,n]$, we have $m_i=p_1^{e_{i,1}}p_2^{e_{i,2}}\cdot\ldots\cdot p_r^{e_{i,1}}$ where $e_{i,j}$ are non-negative integers. We associate each $m_i$ with $a_i\in \mathbb{Z}_{n_1}\oplus \mathbb{Z}_{n_2}\oplus \ldots \oplus \mathbb{Z}_{n_r}$ under the homomorphism:\begin{equation}\label{assoc0}
\begin{split}
\Phi:\{p_1^{e_1}\cdot p_2^{e_2}\cdot\ldots\cdot p_r^{e_r}:\,\,e_i\ge 0,\, e_i\in\mathbb{Z} \}\to\mathbb{Z}_{n_1}\oplus \mathbb{Z}_{n_2}\oplus \ldots \oplus \mathbb{Z}_{n_r},\\
  \Phi(p_1^{e_1}\cdot p_2^{e_2}\cdot\ldots\cdot p_r^{e_r})=\big([e_{1}]_{{}_{n_1}},\,[e_{2}]_{{}_{n_2}},\,\ldots,\,[e_{r}]_{{}_{n_r}}\,\big).
\end{split}
\end{equation}
Hence:
\begin{equation}\label{assoc}
  \Phi(m_i)=\big([e_{i,1}]_{{}_{n_1}},\,[e_{i,2}]_{{}_{n_2}},\,\ldots,\,[e_{i,r}]_{{}_{n_r}}\,\big).
\end{equation}
Thus, we get a sequence $S=a_1a_2\cdot\ldots\cdot a_l$ of elements of the group\\
$\mathbb{Z}_{n_1}\oplus \mathbb{Z}_{n_2}\oplus \ldots \oplus \mathbb{Z}_{n_r}$ of length $l=D(\mathbb{Z}_{n_1}\oplus \mathbb{Z}_{n_2}\oplus \ldots \oplus \mathbb{Z}_{n_r}).$
Therefore, there exists a non-empty zero sum subsequence $T$ of $S$ in $\mathbb{Z}_{n_1}\oplus \mathbb{Z}_{n_2}\oplus \ldots \oplus \mathbb{Z}_{n_r},$ and let $T=a_{j_1}a_{j_2}\cdot\ldots\cdot a_{j_t}.$ That is
\begin{equation}\label{assoc2}
\sum\limits_{i=1}^{t} e_{j_i,k}\equiv 0 \pmod {n_k},
\end{equation}
where $k\in[1,r].$
Consider the subsequence $N$ of $M$ corresponding to $T.$\\
We have $N=(m_{j_1},m_{j_2},\ldots,m_{j_t})$ and by equation \eqref{assoc2}, we get
\begin{equation}\label{assoc3}
  \prod\limits_{m\in N}^{} m=\prod\limits_{k=1}^r {p_k}^{{}^{\sum\limits_{i=1}^t e_{j_i,k}}}=\prod\limits_{k=1}^r\left({p_k}^{n_k}\right)^{l_k},
\end{equation}
for some integers $l_k\ge 0.$
Hence $\prod\limits_{m\in N}^{} m$ is a smooth number over a~set $\{{p_1}^{n_1}, {p_2}^{n_2},\ldots, {p_r}^{n_r}\}$.
By definition of $c(n_1,n_2,\ldots,n_r)$ we get inequality~(\ref{eCD}).
To prove
\begin{equation}\label{eCD2}
c(n_1,n_2,\ldots, n_r)\ge D(\mathbb{Z}_{n_1}\oplus \mathbb{Z}_{n_2}\oplus \ldots \oplus \mathbb{Z}_{n_r}).
\end{equation}
Let $l=c(n_1,n_2,\ldots, n_r)$ and $S=a_1a_2\cdot\ldots\cdot a_l$ be a sequence of elements of $\mathbb{Z}_{n_1}\oplus \mathbb{Z}_{n_2}\oplus \ldots \oplus \mathbb{Z}_{n_r}$ of length $l,$ where $a_i=\big([e_{i,1}]_{{}_{n_1}},\,[e_{i,2}]_{{}_{n_2}},\,\ldots,\,[e_{i,r}]_{{}_{n_r}}\,\big).$
We put $m_i=p_1^{e_{i,1}}\cdot p_2^{e_{i,2}}\cdot\ldots \cdot p_r^{e_{i,r}}.$ The sequence $M=(m_1,m_2,\ldots ,m_l)$ of integers is a sequence of smooth numbers over a set $F.$\\ By definition of $l=c(n_1,n_2,\ldots, n_r),$ there exists a non-empty subsequence $N=(m_{j_1},m_{j_2},\ldots,m_{j_t})$ of $M$ such that:
\begin{equation}
\prod\limits_{m\in N}^{} m=\prod\limits_{k=1}^r\left({p_k}^{n_k}\right)^{l_k},
\end{equation}
for some integers $l_k\ge 0.$ The subsequence $T$ of $S$ corresponding to $N$ will sum up to the identity in $\mathbb{Z}_{n_1}\oplus \mathbb{Z}_{n_2}\oplus \ldots \oplus \mathbb{Z}_{n_r}.$ Therefore (\ref{eCD2}) holds and we obtain (\ref{eCD0}).
\end{proof}
\section{Future work and the non-Abelian case}
The constant $E(G)$ has received a lot of attention in the last $10$ years. For example, one direction went towards weighted zero-sum problems, 
see \cite[Chapter 16]{GRY}. The second direction went towards non-Abelian groups.\\
Let $G$ be any additive finite group. Let $S = (a_1,\ldots,a_n)$ be a sequence over~$G$. 
We say that the sequence $S$ is a zero-sum sequence if there exists a permutation 
$\sigma:[1,n]\to[1,n]$ such that $0=a_{\sigma(1)}+\ldots+ a_{\sigma(n)}.$ 
For a subset $I\subset \mathbb{N},$ let $s_L(G)$ denote the smallest $t\in\mathbb{N}\cup\{0,\infty\}$ 
such that every sequence $S$ over $G$ of length $|S|\ge t$ has a~zero-sum subsequence $S'$ of length $|S'|\in L.$ 
The constants $D(G) := s_\mathbb{N}(G)$ and $E(G) := s_{\{|G|\}}(G)$ are classical invariants in zero-sum theory 
(independently of whether $G$ is Abelian or not). For a given finite group $G$, 
let ${\mathsf d}(G)d(G)$ be the maximal length of a~zero-sum free sequence over $G.$ We call ${\mathsf d}(G)$ the small Davenport constant.
\begin{rem}
In the Abelian case ${\mathsf d}(G)=D(G)-1$. Note that in non-Abelian groups ${\mathsf d}(G)$ can be
strictly smaller than $D(G)-1$, see \cite[Chapter 2]{RG}.
\end{rem}
\begin{thm}\label{zsml12b}
If $G$ is a~finite group of order $|G|$, then
\begin{equation}\label{nac}
{\mathsf d}(G)+m|G|\le {\mathsf E}_m(G)\le {\mathsf E}(G)+(m-1)|G|.
\end{equation}
\end{thm}
\begin{proof}
Let $S=(g_1,g_2,\ldots ,g_{{}_{d(G)}})$ be a~sequence of $d(G)$ non-zero elements in $G.$
Using the definition of $d(G)$, we may assume that $S$ does not contain any non-empty subsequence $S'$ such that $\sigma(S')=0$.
We put
\begin{equation}
T=(a_1,a_2,\ldots ,a_{{}_{d(G)}}, 0,\ldots , 0),
\end{equation}
where $\nu_0(T)=m|G|-1.$\\
We observe that the sequence $T$ not contain $m$ disjoint non-empty subsequences $T_1,T_2,\ldots,T_m$ of $T$ such that $\sigma(T_i)=0$ and $|T_i|=|G|$ for\\ $i\in[1,m].$ This implies that ${\mathsf E}_m(G)>{\mathsf d}(G)+m|G|-1.$ Hence,
\begin{equation}
{\mathsf E}_m(G)\ge {\mathsf d}(G)+m|G|.
\end{equation}
On the other hand, if $S$ is any sequence over $G$ such that $|S|\ge {\mathsf E}(G)+ (m-1)|G|,$ then one can sequentially extract at least $m$ disjoint subsequences $S_1,\ldots,S_m$ of $S,$ such that $\sigma(S_i)=0$ in $G$ and $|S_i|=|G|.$ Thus,
\begin{equation}
{\mathsf E}_m(G)\le {\mathsf E}(G)+(m-1)|G|.
\end{equation}
\end{proof}
\begin{cor}
If $m\to\infty,$ then ${\mathsf E}_m(G)\sim m|G|.$
\end{cor}
\begin{proof}
If $m\to\infty,$ then $\tfrac{{\mathsf E}_m(G}{m}\to |G|$ by the inequality \eqref{nac}.
\end{proof}
We now give an application of theorem \ref{zsml12b}. 
The formula ${\mathsf E}(G)={\mathsf d}(G)+|G|$ was proved for all finite Abelian groups and for some classes of finite non-Abelian groups (see equation \eqref{ED} and \cite{BAS, DH, DH2, OHZ}).
Thus 
$${\mathsf E}_m(G)={\mathsf d}(G)+m|G|$$
holds  for finite groups in the following classes: Abelian groups, nilpotent groups, groups in the form $C_m\ltimes_\varphi C_{mn},$ 
where $m,n\in\mathbb{N},$ dihedral and dicyclic groups and all non-Abelian groups of order $pq$ with $p$ and $q$ prime.
Therefore, the the following conjecture can be proposed:
\begin{hip}
(See \cite[Conjecture 2]{BAS} For any finite group $G,$ we have equation
$${\mathsf E}_m(G)={\mathsf d}(G)+m|G|.$$
\end{hip}


\normalsize \baselineskip=17pt

Maciej Zakarczemny\\
Institute of Mathematics\\
Cracow University of Technology\\
Warszawska 24\\
31--155 Krak\'ow, Poland\\
E-mail: mzakarczemny@pk.edu.pl

\end{document}